\documentclass[letterpaper, 10 pt, conference]{ieeeconf}  % Comment this line out
\IEEEoverridecommandlockouts                              % This command is only
\usepackage{amsmath} % assumes amsmath package installed
\usepackage{amssymb}  % assumes amsmath package installed
\usepackage{amsfonts}
\usepackage{dsfont}
\makeatletter
\newtheorem{lemma}{Lemma}[section]

\newtheorem{corollary}[lemma]{Corollary}
\newtheorem{example}[lemma]{Example}
\newtheorem{definition}[lemma]{Definition}

\newtheorem{proposition}[lemma]{Proposition}
\newtheorem{remark}[lemma]{Remark}

\newcommand{\A}{\mathcal{A}}

\DeclareMathOperator{\diag}{diag} 

\title{\LARGE \bf
Interconnections of input-output Hamiltonian systems with dissipation}

\author{Arjan van der Schaft
\thanks{A.J. van der Schaft is with the Johann Bernoulli Institute for Mathematics and Computer
Science and Jan C. Willems Center for Systems and Control, University of Groningen, PO Box 407, 9700 AK, the
Netherlands,
        {\tt\small A.J.van.der.Schaft@rug.nl}}
}

\begin{document}

\maketitle
\thispagestyle{empty}
\pagestyle{empty}

\begin{abstract}
Recently, negative imaginary and counter-clockwise systems have attracted attention as an interesting class of systems, which is well-motivated by applications. In this paper first the formulation and extension of negative imaginary and counter-clockwise systems as (nonlinear) input-output Hamiltonian systems with dissipation is summarized. Next it is shown how by considering the time-derivative of the outputs a port-Hamiltonian system is obtained, and how this leads to the consideration of alternate passive outputs for port-Hamiltonian systems.
Furthermore, a converse result to positive feedback interconnection of input-output Hamiltonian systems with dissipation is obtained, stating that the positive feedback interconnection of two linear systems is an input-output Hamiltonian system with dissipation if and only if the systems themselves are input-output Hamiltonian systems with dissipation. This implies that the Poisson and resistive structure matrices can be redefined in such a way that the interaction between the two systems only takes place via the coupling term in the Hamiltonian of the interconnected system. Subsequently, it is shown how network interconnection of (possibly nonlinear) input-output Hamiltonian systems with dissipation results in another input-output Hamiltonian system with dissipation, and how this leads to a stability analysis of the interconnected system in terms of the Hamiltonians and output mappings of the systems associated to the vertices, as well as the topology of the network.
\end{abstract}

\section{Properties of input-output Hamiltonian systems with dissipation}

Consider a linear system
\begin{equation}\label{linsystem}
\begin{array}{rcl}
\dot{x} & = & Ax + Bu, \quad x \in \mathbb{R}^n, u \in \mathbb{R}^m\\[2mm]
y & = & Cx + Du, \quad y \in \mathbb{R}^m 
\end{array}
\end{equation}
with transfer matrix $G(s)=C(Is-A)^{-1}B + D$.
In \cite{lanzonpetersen1,lanzonpetersen2,xiong} $G(s)$ was called {\it negative imaginary}\footnote{The terminology 'negative imaginary', stems, similarly to 'positive real', from the Nyquist plot interpretation for single-input single-output systems \cite{lanzonpetersen1,angeli1}.} if $D=D^T$ and the transfer matrix $H(s):=s(G(s)-D)$ is {\it positive real}. In \cite{angeli0,angeli1,angeli2} the same notion (also in a nonlinear context) was coined as {\it counter-clockwise input-output dynamics}.

In \cite{lanzonpetersen1, xiong} it was shown that a minimal system (\ref{linsystem}) has negative imaginary transfer matrix if and only if $D=D^T$ and there exists an $n \times n$ symmetric matrix $Q >0$ such that
\begin{equation}\label{P}
A^TQ + QA \leq 0, \quad B = -AQ^{-1}C^T
\end{equation}
Subsequently in \cite{cdc2011} it was shown, by decomposing $AQ^{-1}$ into its symmetric and skew-symmetric part, that a minimal system (\ref{linsystem}) has negative imaginary transfer matrix if and only if it can be written as
\begin{equation}\label{linearham}
\begin{array}{rcl}
\dot{x} & = & [J-R](Qx - C^Tu) \\[2mm]
y & = & Cx + Du, \quad D=D^T
\end{array}
\end{equation}
for $Q,J,R$ satisfying 
\begin{equation}\label{linearhamprop}
Q=Q^T, \; J=-J^T, \; R=R^T \geq 0
\end{equation}
with $Q >0$.

Any system (\ref{linearham}) satisfying (\ref{linearhamprop}) (not necessarily $Q>0$) was called in \cite{cdc2011} an {\it input-output Hamiltonian system with dissipation} (IOHD system), with {\it Hamiltonian} function $\frac{1}{2}x^TQx$. The skew-symmetric matrix $J$ defines a {\it Poisson structure} matrix, while $R$ is called the {\it resistive structure} matrix; see \cite{bordeaux, vanderschaftmaschkearchive, dalsmo, vanderschaftbook, ortega, NOW} for the closely related port-Hamiltonian case (as discussed below).

The definition of an IOHD system \eqref{linearham} was extended in \cite{cdc2011} to the nonlinear case as follows. For clarity of exposition, we will throughout only consider the case without feedthrough terms and with affine dependence on $u$; see \cite{cdc2011}, as well as Remark 2.3 below, for the general nonlinear case.
\begin{definition}\label{def:nonlinearham}
A system described in local coordinates $x=(x_1,\cdots,x_n)$ for some $n$-dimensional state space manifold $\mathcal{X}$ as\footnote{For a function $H: \mathbb{R}^n \to \mathbb{R}$ we denote by $\frac{\partial H}{\partial x}(x)$ the $n$-dimensional column vector of partial derivatives of $H$. For a mapping $C: \mathbb{R}^n \to \mathbb{R}^m$ we denote by $\frac{\partial C^T}{\partial x}(x)$ the $n \times m$ matrix whose $j$-th column consists of the partial derivatives of the $j$-th component function $C_j$.}
\begin{equation}\label{nonlinearham}
\begin{array}{rcl}
\dot{x} & = & [J(x)-R(x)]\left(\frac{\partial H}{\partial x}(x) - \frac{\partial C^T}{\partial x}(x)u \right),  \quad u \in \mathbb{R}^m \\[3mm]
y & = & C(x), \quad y \in \mathbb{R}^m
\end{array}
\end{equation}
where the $n \times n$ matrices $J(x),R(x)$, depending smoothly on $x$, satisfy 
\begin{equation}\label{nonlinearhamprop}
J(x)=-J^T(x), \; R(x)=R^T(x) \geq 0,
\end{equation}
is called a nonlinear IOHD system, with Hamiltonian $H: \mathcal{X} \to \mathbb{R}$ and output mapping $C: \mathcal{X} \to \mathbb{R}^m$. 
\end{definition}

\smallskip

This definition is a generalization of the definition of an affine {\it input-output Hamiltonian system} as originally proposed in \cite{brockett} and studied in e.g. \cite{vanderschaftMST, vanderschaft1,crouch}. In fact, it reduces to this definition in case $R=0$ (no dissipation) and $J$ defines a {\it symplectic form} (in particular, has full rank). 

\smallskip
The time-evolution of the Hamiltonian of a nonlinear IOHD system \eqref{nonlinearham} is computed as (exploiting skew-symmetry of $J(x)$)
\begin{equation}\label{balance1}
\begin{array}{l}
\frac{d}{dt}H  =  \left(\frac{\partial H}{\partial x}(x)\right)^T [J(x) -R(x)] \left(\frac{\partial H}{\partial x}(x) - \frac{\partial C^T}{\partial x}(x)u \right) =\\[3mm]
  
  \!\! - (\frac{\partial H}{\partial x}(x))^T R(x)\frac{\partial H}{\partial x}(x)\! - \! (\frac{\partial H}{\partial x}(x))^T [J(x) -R(x)]\frac{\partial C^T \!\!}{\partial x}(x)u 
 \end{array}
\end{equation}
Furthermore, the time-differentiated output of \eqref{nonlinearham} is
\begin{equation}\label{nonlinearham2}
z:=\dot{y} = (\frac{\partial C^T}{\partial x}(x))^T [J(x) -R(x)] \left(\frac{\partial H}{\partial x}(x) - \frac{\partial C^T}{\partial x}(x)u \right)
\end{equation}
Using $u^T(\frac{\partial C^T}{\partial x}(x))^TJ(x) \frac{\partial C^T}{\partial x}(x)u =0$ it is verified that the expression \eqref{balance1} can be rewritten as (leaving out arguments $x$)
\begin{equation}\label{balance}
\begin{array}{l}
\frac{d}{dt}H  =  u^Tz - (\frac{\partial H}{\partial x} - u^T \frac{\partial C^T}{\partial x})^TR(\frac{\partial H}{\partial x} - u^T \frac{\partial C^T}{\partial x}) = \\[3mm]
u^Tz - \begin{bmatrix} (\frac{\partial H}{\partial x})^T & \!\! u^T \end{bmatrix}
\! \begin{bmatrix} R & -R\frac{\partial C^T}{\partial x} \\ - (\frac{\partial C^T}{\partial x})^TR & (\frac{\partial C^T}{\partial x})^TR\frac{\partial C^T}{\partial x} \end{bmatrix} 
\! \begin{bmatrix} \frac{\partial H}{\partial x} \\[2mm] u \end{bmatrix}  \\[5mm]
\quad \quad \leq u^Tz
\end{array}
\end{equation}
This immediately shows {\it passivity} with respect to the output $z$ defined by \eqref{nonlinearham2} if additionally the Hamiltonian $H$ is bounded from below. 
In fact, the system \eqref{nonlinearham} with output $y_{PH}=z$ for a general Hamiltonian $H$ is an input-state-output {\it port-Hamiltonian system} \cite{vanderschaftbook, vanderschaftmaschkearchive, dalsmo, ortega}, of the general form \cite{NOW}
\begin{equation}\label{ph}
\begin{array}{rcl}
\dot{x} \!& = & \![J(x)-R(x)] \frac{\partial H}{\partial x}(x) + [G(x) - P(x)] u, \; x \in \mathcal{X} \\[3mm]
y_{PH} \!& = & \![G(x) + P(x)]^T\frac{\partial H}{\partial x}(x) + [M(x) + S(x)] u
\end{array}
\end{equation}
with
\[
\begin{bmatrix} R(x) & P(x) \\ P^T(x) & S(x) \end{bmatrix} \mbox{ symmetric and } \geq 0 ,
\]
and $J(x)$ and $M(x)$ skew-symmetric. This can be seen by equating
\begin{equation}\label{ph-IOHD}
\begin{array}{l}
G(x) = - J(x) \frac{\partial C^T}{\partial x}(x), \; P(x) = - R(x) \frac{\partial C^T}{\partial x}(x), \\[2mm]
S(x) = (\frac{\partial C^T}{\partial x}(x))^TR(x) \frac{\partial C^T}{\partial x}(x), \\[2mm]
M(x) = - (\frac{\partial C^T}{\partial x}(x))^TJ(x) \frac{\partial C^T}{\partial x}(x)
\end{array}
\end{equation}
This leads to
\begin{proposition}
Given the IOHD system \eqref{nonlinearham}. Then its dynamics together with differentiated output $z =\dot{y}$ defined by \eqref{nonlinearham2} is an input-state-output port-Hamiltonian system of the form \eqref{ph} with $y_{PH}=z$. Conversely, given a port-Hamiltonian system \eqref{ph}, then there exists an IOHD system with the same dynamics and output $y=C(x)$, $C: \mathcal{X} \to \mathbb{R}^m$, such that $\dot{y}= y_{PH}$ if and only if $C$ satisfies \eqref{ph-IOHD}.
\end{proposition}

\smallskip

Note that the conditions \eqref{ph-IOHD} can be interpreted as (generalized) {\it integrability conditions} on $G(x), M(x)$ and $P(x),S(x)$. Indeed, for the special case of a basic input-state-output port-Hamiltonian system
\begin{equation}\label{ph1}
\begin{array}{rcl}
\dot{x} & = & [J(x)-R(x)] \frac{\partial H}{\partial x}(x) + g(x) u \\[3mm]
y_{PH} & = & g^T(x) \frac{\partial H}{\partial x}(x)
\end{array}
\end{equation}
corresponding to $P=0,S=0,M=0$ and $g=G$, the conditions \eqref{ph-IOHD} reduce to
\begin{equation}
\begin{array}{l}
g(x) = G(x)= - J(x) \frac{\partial C^T}{\partial x}(x), \\[2mm]
R(x) \frac{\partial C^T}{\partial x}(x)=0, \quad (\frac{\partial C^T}{\partial x}(x))^TJ(x) \frac{\partial C^T}{\partial x}(x) =0
\end{array}
\end{equation}
The first line implies that the columns $g_j, j=1, \cdots,m$ of the input matrix $g(x)$ are Hamiltonian vector fields with Hamiltonians $-C_1, \cdots, -C_m$. For $J$ corresponding to a symplectic structure, there exist locally such functions $C_1, \cdots,C_m$ if and only if the vector fields $g_j$ leave the symplectic structure invariant \cite{abraham, vanderschaftMST, vanderschaft1}.
\begin{example}
Consider the linear mechanical system as considered in \cite{bordeaux}, consisting of an alternating mass-spring-mass-spring system, where the first input $u_1$ is the velocity of the right-hand side of the second spring, and the first input $u_2$ is the force on the left mass. Furthermore, differently from \cite{bordeaux}, there are dampers with damping coefficients $d_1,d_2$ parallel to the two springs with spring constants $k_1,k_2$. Denoting the masses by $m_1,m_2$, the elongations of the springs by $q_{12}, q_{20}$, and the momenta of the masses by $p_1,p_2$, the dynamical equations are given as
\begin{equation}
\begin{array}{rcl}
\dot{q}_{12} & = & \frac{p_1}{m_1} -  \frac{p_2}{m_2} \\[2mm]
\dot{q}_{20} & = &  \frac{p_2}{m_2} - u_1 \\[2mm]
\dot{p}_1 & = & - k_1q_{12} + u_2\\[2mm]
\dot{p}_2 & = &  k_1q_{12} - k_2q_{20}
\end{array}
\end{equation}
This can be written as the IOHD system %state vector $x=(q_{12},q_{20},p_1,p_2)$
\begin{equation}
\begin{array}{l}
\begin{bmatrix} 
\dot{q}_{12} \\ \dot{q}_{20} \\ p_1 \\p_2 \\ 
\end{bmatrix} \! = 
[\begin{bmatrix}
0 & 0 &1 & -1 \\
0 & 0 & 0 & 1 \\
-1 & 0 & 0 & 0 \\
1 & -1 & 0 & 0
\end{bmatrix}
 \!- \!
\begin{bmatrix}
0 & 0 &0 & 0 \\
0 & 0 & 0 & 0 \\
0 & 0 & d_1 & -d_1 \\
0 & 0 & -d_1 & d_1 + d_2
\end{bmatrix}] \\[10mm]
\quad \quad \left(\begin{bmatrix} 
\frac{\partial H}{\partial q_{12}} \\[2mm] \frac{\partial H}{\partial q_{20}} \\[2mm] \frac{\partial H}{\partial p_1} \\[2mm] \frac{\partial H}{\partial p_2} \end{bmatrix}
-
\begin{bmatrix}
\frac{\partial C_1}{\partial q_{12}} \\[2mm] \frac{\partial C_1}{\partial q_{20}} \\[2mm] \frac{\partial C_1}{\partial p_1} \\[2mm] \frac{\partial C_1}{\partial p_2} \end{bmatrix} u_1
-
\begin{bmatrix}
\frac{\partial C_2}{\partial q_{12}} \\[2mm] \frac{\partial C_2}{\partial q_{20}} \\[2mm] \frac{\partial C_2}{\partial p_1} \\[2mm] \frac{\partial C_2}{\partial p_2} \end{bmatrix} u_2
\right) \, ,
\end{array}
\end{equation}
with Hamiltonian $H(q_{12},q_{20},p_1,p_2)$ and outputs $ y_1=C_1(q_{12},q_{20},p_1,p_2), y_2 =C_2(q_{12},q_{20},p_1,p_2)$ given by
\begin{equation}
\begin{array}{rcl}
H(q_{12},q_{20},p_1,p_2) & = & \frac{1}{2} k_1q^2_{12} + \frac{1}{2} k_2q^2_{20} + \frac{p_1^2}{2m_1} + \frac{p_2^2}{2m_2}  \\[2mm]
C_1(q_{12},q_{20},p_1,p_2) & = & p_1 + p_2 - d_2 q_{20} \\[2mm]
C_2(q_{12},q_{20},p_1,p_2) & = & q_{12} + q_{20}
\end{array}
\end{equation}
The resulting port-Hamiltonian system with respect to the differentiated outputs $y_{PH1}=\dot{y}_1, y_{PH2}=\dot{y}_2$ is given in the form \eqref{ph}, where $G,P,S,M$ are the constant matrices given as
\begin{equation}
\begin{array}{rccccl}
G &= & \begin{bmatrix} 0 & 0 \\ -1 & 0 \\0 & 1 \\ -d_2 & 0\end{bmatrix}, 
& P &= &\begin{bmatrix} 0 & 0 \\ 0 & 0 \\0 & 0 \\ -d_2 & 0\end{bmatrix}, \\[8mm]
S &= &\begin{bmatrix} d_2 & 0 \\ 0 & 0 \end{bmatrix}, 
& M &= &\begin{bmatrix} 0 & 1 \\ -1 & 0\end{bmatrix}
\end{array}
\end{equation}
Note that even for the case $d_2=0$ the resulting port-Hamiltonian system is not anymore of the basic form \eqref{ph1}, due to the presence of the feedthrough matrix $M$, as already noticed (from a different point of view) in \cite{bordeaux}. 
\end{example}

\smallskip

Obviously, we can consider alternate {\it passive outputs} for the port-Hamiltonian system \eqref{ph}. (An output $\hat{y}_{PH}$ is called a passive output for \eqref{ph} if $\frac{d}{dt}H \leq u^T\hat{y}_{PH}$.)The simplest choice is to define the new passive output $\hat{y}_{PH} := g^T(x)\frac{\partial H}{\partial x}(x)$, with $g(x)=G(x)-P(x)$, resulting in a new input-state-output port-Hamiltonian system of the basic form \eqref{ph1}. In general, given the port-Hamiltonian system \eqref{ph} all outputs
\begin{equation}
\hat{y}_{PH}  =  [\widehat{G}(x) + \widehat{P}(x)]^T\frac{\partial H}{\partial x}(x) + [\widehat{M}(x) + \widehat{S}(x)] u,
\end{equation}
with $\widehat{M}(x)=-\widehat{M}^T(x)$ and $\widehat{G}(x), \widehat{P}(x),\widehat{S}(x)$ satisfying
\begin{equation}
\begin{array}{l}
\widehat{G}(x)-\widehat{P}(x) = G(x)-P(x) \\[2mm]
\begin{bmatrix} R(x) & \widehat{P}(x) \\ \widehat{P}^T(x) & \widehat{S}(x) \end{bmatrix} \geq 0 
\end{array}
\end{equation}
define alternate passive outputs, corresponding to alternate port-Hamiltonian systems. Notice, however, that these alternate outputs can{\it not} always be integrated to outputs of an IOHD system. 
Such new passive outputs were recently used in \cite{borja} for IDA-PBC control, continuing on e.g. \cite{jeltsema}, see also \cite{venkatraman}.

Still a larger class of passive outputs can be obtained by allowing for different $J,R$ and $H$ in \eqref{ph}, in such a way that the dynamics of \eqref{ph} remains the same.
%\begin{example}\label{ex:1}
\subsection{Mechanical systems with collocated sensors and actuators}
IOHD systems show up naturally in many applications; see e.g. \cite{lanzonpetersen1, lanzonpetersen2,angeli2, bernstein, brockett, vanderschaftMST, vanderschaft1}. A clear example are mechanical systems with co-located position sensors and force actuators, which in the linear case are represented as the IOHD systems (with $q$ denoting the position vector and $p$ the momentum vector) 
\begin{equation}\label{hamstate}
\begin{array}{rcl}
\begin{bmatrix} \dot{q} \\ \dot{p} \end{bmatrix} & = & \begin{bmatrix} 0_n & I_n \\ -I_n & -D \end{bmatrix} \begin{bmatrix} K & N \\ N^T & M^{-1} \end{bmatrix} \begin{bmatrix} q \\ p \end{bmatrix} + 
\begin{bmatrix} 0 \\ L^T \end{bmatrix} u \\[4mm]
y & = & Lq ,
\end{array}
\end{equation}
where $D \geq 0$ is the damping matrix; defining the resistive structure matrix $R=\diag (0, D)$. 
Usually $N=0$ (no 'gyroscopic forces'), in which case the Hamiltonian (total energy) is given as
\begin{equation}
H(q,p) = \frac{1}{2}q^TKq + \frac{1}{2}p^TM^{-1}p,
\end{equation}
where the first term is the total potential energy (with $K$ the compliance matrix), and the second term is the kinetic energy (with $M$ the mass matrix). 

For $N=0$ \eqref{hamstate} can be rewritten into the equivalent second-order form
\begin{equation}\label{secondorder}
M\ddot{q} + D\dot{q} + Kq=L^Tu, \quad y=Lq
\end{equation} 
%\end{example}

\section{Positive feedback interconnection of IOHD systems}

Just like the {\it negative} feedback interconnection of passive (or port-Hamiltonian) systems results in a passive (respectively, port-Hamiltonian) system, the {\it positive} feedback interconnection of IOHD systems results in another IOHD system; see \cite{angeli0, angeli1} for the counter-clockwise case. Indeed, the positive feedback interconnection
\begin{equation}\label{posfeedback}
u_1 = y_2 + e_1, \quad u_2 = y_1 + e_2,
\end{equation}
with $e_1,e_2$ two external inputs, of two {\it linear} IOHD systems
\begin{equation}\label{twoioham}
\begin{array}{rcl}
\dot{x}_i & = & [J_i-R_i]\left(Q_ix_i - C_i^Tu_i \right) \\[2mm]
y_i & = & C_ix_i, \quad i=1,2,
\end{array}
\end{equation}
with equal number of inputs and outputs can be seen \cite{cdc2011} to result in the IOHD system
\begin{equation}\label{interham}
\begin{array}{l}
\begin{bmatrix} \dot{x}_1 \\ \dot{x}_2 \end{bmatrix}  = 
\left( \begin{bmatrix} J_1 & 0 \\ 0 & J_2 \end{bmatrix} - 
\begin{bmatrix} R_1 & 0 \\ 0 & R_2 \end{bmatrix}\right) \\[5mm]
\quad  \quad \left(
\begin{bmatrix} Q_1  & -C_1^TC_2 \\ -C_2^TC_1 & Q_2  \end{bmatrix}
\begin{bmatrix} x_1 \\ x_2 \end{bmatrix}\right. - \\[5mm]
 \quad \quad \quad \quad \quad \quad - \left.
\begin{bmatrix} C_1^T & 0 \\ 0 & C_2^T \end{bmatrix} 
\begin{bmatrix} e_1 \\ e_2 \end{bmatrix} \right) \\[5mm]
\begin{bmatrix} y_1 \\ y_2 \end{bmatrix}  =  
\begin{bmatrix} C_1 & 0 \\ 0 & C_2 \end{bmatrix}\begin{bmatrix} x_1 \\ x_2 \end{bmatrix},
\end{array}
\end{equation}
with {\it interconnected Hamiltonian} given as
\begin{equation}\label{interham1}
H_{\mathrm{int}}(x_1,x_2) := \frac{1}{2}x_1^TQ_1 x_1 + \frac{1}{2}x_2^TQ_2 x_2 - x_1^TC_1^TC_2x_2
\end{equation}
Hence the stability of the interconnected system can be characterized in terms of the interconnected Hamiltonian (\ref{interham1}) as follows.
\begin{proposition}\cite{cdc2011}
Consider two IOHD systems. The interconnected IOHD system (\ref{interham}) is stable having no eigenvalue at zero if the interconnected Hamiltonian (\ref{interham1}) has a strict minimum at the origin $(x_1,x_2) = (0,0)$. Conversely, if (\ref{interham}) is asymptotically stable, then the interconnected Hamiltonian (\ref{interham1}) has a strict minimum at the origin $(x_1,x_2) = (0,0)$.
\end{proposition}

\smallskip

In (\cite{lanzonpetersen1}, Theorem 5) it has been shown that $\mathcal{Q}$, with $Q_1>0,Q_2>0$ is positive definite, if and only 
\begin{equation}\label{schur}
\lambda_{\mathrm{max}}\left(C_1Q^{-1}_1C^T_1 \cdot C_2Q_2^{-1}C^T_2\right) < 1 ,
\end{equation}
where $\lambda_{\mathrm{max}}(K)$ denotes the maximal eigenvalue of a symmetric matrix $K$. An easy proof follows by the fact that $\mathcal{Q} >0$ if and only if $Q_1 > 0$ and $Q_2 - C_2^TC_1Q_1^{-1}C_1^TC_2 >0$. This last inequality is equivalent to $Q_2^{-\frac{1}{2}}C_2^TC_1Q_1^{-1}C_1^TC_2Q_2^{-\frac{1}{2}} < I$, and thus to
\[
C_1Q_1^{-1}C_1^T\cdot C_2 Q_2^{-\frac{1}{2}} Q_2^{-\frac{1}{2}}C_2^T < I ,
\]
which is equivalent to \eqref{schur}.

This allows for the following interpretation. The {\it dc-gain} of an IOHD system (\ref{linearham}) with $D=0$ is given by the expression
\begin{equation}\label{dcgain}
-CA^{-1}B= CQ^{-1}C^T
\end{equation}
Hence the interconnected IOHD system (\ref{interham}) resulting from the positive feedback interconnection of two stable IOHD systems ($Q_1>0,Q_2>0$) is stable having no eigenvalue at zero if and only if {\it the dc loop gain is less than unity.} This can be regarded as a rephrasement of the fundamental result concerning the stability of the positive feedback interconnection of two systems with negative imaginary transfer matrices, as obtained in \cite{angeli1} for the SISO case with $D=0$ and in \cite{lanzonpetersen1} for the general MIMO case. 

In the case of positive feedback interconnection of two IOHD systems in second-order form \eqref{secondorder} this amounts to the following corollary.
\begin{corollary} Consider two systems \eqref{secondorder} with $M_i>0,K_i>0, i=1,2$. Then the positive feedback interconnection results in the second-order system
\begin{equation}
\begin{array}{l}
\begin{bmatrix} M_1 & 0 \\ 0 & M_2 \end{bmatrix} \begin{bmatrix} \ddot{q}_1 \\ \ddot{q}_2 \end{bmatrix} +
\begin{bmatrix} D_1 & 0 \\ 0 & D_2 \end{bmatrix} \begin{bmatrix} \dot{q}_1 \\ \dot{q}_2 \end{bmatrix} + \\[4mm]
\begin{bmatrix} K_1 & -L_1^TL_2 \\ -L_2^TL_1 & K_2 \end{bmatrix} \begin{bmatrix} {q}_1 \\ {q}_2 \end{bmatrix} =
\begin{bmatrix} L^T_1 & 0 \\ 0 & L^T_2 \end{bmatrix} \begin{bmatrix} e_1 \\ e_2 \end{bmatrix} \\[4mm]
\begin{bmatrix} y_1 \\ y_2 \end{bmatrix} = \begin{bmatrix} L_1 & 0 \\ 0 & L_2 \end{bmatrix} \begin{bmatrix} q_1 \\ q_2 \end{bmatrix},
\end{array}
\end{equation}
which is stable without eigenvalue at zero if and only if
\begin{equation}\label{schur1}
\lambda_{\mathrm{max}}\left(L_1K^{-1}_1L^T_1 \cdot L_2K_2^{-1}L^T_2 \right) < 1
\end{equation}
\end{corollary}

\medskip
The positive feedback interconnection of a second-order 'plant' system \eqref{secondorder} with a second-order 'controller' system is known in the literature, see e.g. \cite{goh,friswell}, as {\it positive position feedback}. It has the advantage of being insensitive to spillover, and has favorable other robustness properties, see e.g. \cite{friswell} for a discussion. Note that by \eqref{schur1} the stability of the closed-loop system only depends on the potential energies of the 'plant' and 'controller' system and on the matrices $L_1,L_2$. In particular the stability does depend not on the damping matrices $D_1,D_2$.

\smallskip

Similar to the linear case it can be seen \cite{cdc2011} that the positive feedback interconnection of two nonlinear IOHD systems\footnote{For the generalization to general nonlinear IOHD systems see \cite{cdc2011}.} \eqref{nonlinearham}, indexed by $i=1,2,$ 
%\begin{equation}\label{nonlinearham-i}
%\begin{array}{rcl}
%\dot{x}_i & = & (J_i(x_i)-R_i(x_i))[\frac{\partial H_i}{\partial x_i}(x_i) - \frac{\partial C_i^T}{\partial x_i}(x_i)u_i],  \quad u_i \in \mathbb{R}^m \\[2mm]
%y_i & = & C_i(x_i), \quad y \in \mathbb{R}^m
%\end{array}
%\end{equation}
results in the nonlinear IOHD system 
\begin{equation}\label{nonlinearinterham}
\begin{array}{l}
\begin{bmatrix} \dot{x}_1 \\ \dot{x}_2 \end{bmatrix}  = 
\left( \begin{bmatrix} J_1(x_1) & 0 \\ 0 & J_2(x_2) \end{bmatrix} - 
\begin{bmatrix} R_1(x_1) & 0 \\ 0 & R_2(x_2) \end{bmatrix}\right) \\[5mm]
\quad \left(
\begin{bmatrix} \frac{\partial H_{\mathrm{int}}}{\partial x_1}(x_1, x_2) \\ \frac{\partial H_{\mathrm{int}}}{\partial x_2}(x_1, x_2)\end{bmatrix}   -
\begin{bmatrix} \frac{\partial C_1^T}{\partial x_1}(x_1) & 0 \\ 0 & \frac{\partial C_2^T}{\partial x_2}(x_2) \end{bmatrix} 
\begin{bmatrix} e_1 \\ e_2 \end{bmatrix} \right) \\[7mm]
\begin{bmatrix} y_1 \\ y_2 \end{bmatrix}  =  
\begin{bmatrix} C_1(x_1)  \\ C_2(x_2) \end{bmatrix},
\end{array}
\end{equation}
with interconnected Hamiltonian $H_{\mathrm{int}}$ given by
\begin{equation}\label{nonlinearinterham1}
H_{\mathrm{int}}(x_1,x_2) := H_1(x_1) + H_2(x_2) - C_1^T(x_1)C_2(x_2)
\end{equation}
(compare with \cite{angeli1}, Theorem 6).
Like in the linear case, the stability properties of the interconnected system are determined by $H_{\mathrm{int}}$.

\begin{remark}
A {\it general nonlinear IOHD system} is defined in \cite{cdc2011} as a system of the form
\begin{equation}\label{nonlinearhamgeneral}
\begin{array}{rcl}
\dot{x} & = & (J(x)-R(x))\frac{\partial H}{\partial x}(x,u),  \quad u \in \mathbb{R}^m, \\[3mm]
y & = & - \frac{\partial H}{\partial u}(x,u), \quad y \in \mathbb{R}^m,
\end{array}
\end{equation}
for some function $H(x,u)$, with $R(x),J(x)$ satisfying (\ref{nonlinearhamprop}). Furthermore, it is shown in \cite{cdc2011} how also the positive feedback interconnection of two general IOHD systems results in another general IOHD system, provided some transversality conditions are met.
A particular case of \eqref{nonlinearhamgeneral} is a {\it static} system
\begin{equation}\label{hamstatic}
y  =  - \frac{\partial P}{\partial u}(u), \quad u, y \in \mathbb{R}^m,
\end{equation}
for some function $P: \mathbb{R}^m \to \mathbb{R}$. The positive feedback interconnection of an IOHD system \eqref{nonlinearham} with such a static IOHD system \eqref{hamstatic} results in the IOHD system \eqref{nonlinearham}, with modified Hamiltonian given as
\begin{equation}
H_{cl}(x) :=H(x) + P(C(x))
\end{equation}
Conversely, it can be shown \cite{nvds} that any static output feedback applied to \eqref{nonlinearham} will result in an IOHD system with respect to the same $J(x),R(x)$ if and only it corresponds to positive feedback interconnection with a static IOHD system \eqref{hamstatic}  for some $P$.
\end{remark}

\smallskip

We note that the Poisson and resistive structure of the interconnected system \eqref{nonlinearinterham} is the direct sum of the respective structures of the two component systems. This is opposite to the case of the negative feedback interconnection of two port-Hamiltonian systems \cite{vanderschaftbook, NOW}, where the Poisson structure matrix of the interconnected contains an additional {\it coupling term}, and where, on the other hand, the Hamiltonian of the interconnected system is just the sum of the Hamiltonians of the two component systems; see also \cite{cdc2011}.

\section{A converse result}
In this section we will show that not only the positive feedback interconnection of two IOHD systems results in another IOHD system, but that, at least in the linear case, {\it also the converse holds}: if the positive feedback interconnection of two arbitrary linear systems is an IOHD system (with inputs $e_1,e_2$ and outputs $y_1,y_2$), then the two systems are necessarily IOHD as well. 
%Furthermore, their Poisson structure and dissipation matrices are obtained by decomposition of the Poisson structure matrix and resistive structure matrix of the interconnected IOHD system, while also their Hamiltonians are directly inferred from the Hamiltonian of the interconnected system. 

This result can be seen as an analog to the converse result obtained for the {\it negative} feedback interconnection of {\it passive} systems in \cite{cdc2011a}.
\begin{proposition}\label{prop:converse}
Consider two linear systems $\Sigma_1=(A_1,B_1,C_1), \Sigma_2=(A_2,B_2,C_2)$ with equal input and output dimensions. Suppose the positive feedback interconnection \eqref{posfeedback} of $\Sigma_1$ and $\Sigma_2$ results in a minimal system, with inputs $e_1,e_2$ and outputs $y_1,y_2$, that is IOHD. Then also $\Sigma_1, \Sigma_2$ are IOHD systems.
\end{proposition}
\begin{proof}
Without loss of generality assume that $C_1$ and $C_2$ have full row rank. Since the positive feedback interconnection of $\Sigma_1$ and $\Sigma_2$ is a minimal IOHD system, there exists an invertible matrix $P=Q^{-1}$, cf. \eqref{P}, such that
\begin{equation}\label{4}
\begin{array}{l}
\begin{bmatrix} P_{11} & P_{12} \\ P_{12}^T & P_{22} \end{bmatrix}
\begin{bmatrix} A_1^T & C_1^TB_2^T \\ C_2^T B_1^T & A_2^T \end{bmatrix} + \\[4mm]
\quad \quad \quad
\begin{bmatrix} A_1 & B_1 C_2 \\ B_2 C_1 & A_2 \end{bmatrix} 
\begin{bmatrix} P_{11} & P_{12} \\ P_{12}^T & P_{22} \end{bmatrix} \leq 0
\end{array}
\end{equation}
and 
\begin{equation}\label{5}
\begin{bmatrix} B_1 & 0 \\ 0 & B_2 \end{bmatrix} = - \begin{bmatrix} A_1 & B_1 C_2 \\ B_2 C_1 & A_2 \end{bmatrix} 
\begin{bmatrix} P_{11} & P_{12} \\ P_{12}^T & P_{22} \end{bmatrix} \begin{bmatrix} C^T_1 & 0 \\ 0 & C^T_2 \end{bmatrix}
\end{equation}
Then \eqref{5} is written out as
\begin{equation}\label{5a}
\begin{array}{rcl}
B_1 & = & -A_1P_{11}C_1^T - B_1C_2P_{12}^TC_1^T \\[2mm]
B_2 & = & -A_2P_{22}C_2^T - B_2C_1P_{12}C_2^T 
\end{array}
\end{equation}
\begin{equation}\label{5b}
\begin{array}{l}
A_1P_{12}C_2^T + B_1C_2P_{22}^TC_2^T = 0\\[2mm]
A_2P_{12}^TC_1^T + B_2C_1P_{11}C_1^T =0
\end{array}
\end{equation}
Then the equations \eqref{5b} yield
\begin{equation}\label{6}
\begin{array}{rcl}
B_1 & = & -A_1P_{12}C_2^T(C_2P_{22}C_2^T)^{-1} \\[2mm]
B_2 & = & -A_2P_{12}^TC_1^T(C_1P_{11}C_1^T)^{-1}
\end{array}
\end{equation}
Substituted in the righthand sides of \eqref{5a} this yields
\begin{equation}\label{7}
\begin{array}{rcl}
B_1 & = & - A_1(P_{11} -P_{12}C_2^T(C_2P_{22}C_2^T)^{-1}P_{12}^T)C_1^T \\[2mm]
B_2 & = & - A_2(P_{22} -P_{12}^TC_1^T(C_1P_{11}C_1^T)^{-1}P_{12})C_2^T
\end{array}
\end{equation}
Define subsequently
\begin{equation}
\begin{array}{rcl}
P_1 := P_{11} -P_{12}C_2^T(C_2P_{22}C_2^T)^{-1}P_{12}^T \\[2mm]
P_2 := P_{22} -P_{12}^TC_1^T(C_1P_{11}C_1^T)^{-1}P_{12}
\end{array}
\end{equation}
Note that the first expression corresponds to a Schur complement of
\begin{equation}%\label{9}
\begin{bmatrix} P_{11} & P_{12}C_2^T \\ c_2P_{12}^T & C_2P_{22}C_2^T \end{bmatrix} =
\begin{bmatrix} I & 0 \\ 0 & C_2\end{bmatrix} \begin{bmatrix}P_{11} & P_{12} \\ P_{12}^T & P_{22} \end{bmatrix}
\begin{bmatrix} I & 0 \\ 0 & C_2^T\end{bmatrix}
\end{equation}
Consequently, $P_1$ is invertible, and positive definite if $P$ is positive definite. Similarly, $P_1$ is invertible, and positive definite if $P$ is positive definite.

Finally, write out the $(1,1)$ and $(2,2)$ block of \eqref{4} as
\begin{equation}\label{9}
\begin{array}{rcl}
P_{11}A_1^T + A_1P_{11} + P_{12}C_2^TB_1^T + B_1C_2P_{12}^T & \leq & 0\\[2mm]
P_{22}A_2^T + A_2P_{22} + P_{12}^TC_1^TB_2^T + B_2C_1P_{12} & \leq & 0
\end{array}
\end{equation}
Substitution of the first line of \eqref{6} into the first line of \eqref{9} yields
\begin{equation}\label{10}
\begin{array}{l}
P_{11} A_1^T + A_1P_{11} - P_{12}C_2^T(C_2P_{22}C_2^T)^{-1}C_2P_{12}^TA_1^T - \\[2mm]
\quad \quad A_1P_{12}C_2^T(C_2P_{22}C_2^T)^{-1}C_2P_{12}^T \leq 0,
\end{array}
\end{equation}
which can be rewritten as
\begin{equation}\label{11}
\begin{array}{l}
[P_{11} - P_{12}C_2^T(C_2P_{22}C_2^T)^{-1}C_2P_{12}^T] A_1^T + \\[2mm]
\quad \quad A_1[P_{11} - P_{12}C_2^T(C_2P_{22}C_2^T)^{-1}C_2P_{12}^T] \leq 0,
\end{array}
\end{equation}
that is, $P_1A_1^T + A_1P_1 \leq 0$. Similar derivation holds for $P_2$, altogether resulting in
\begin{equation}
P_1A_1^T + A_1P_1 \leq 0, \quad P_2A_2^T + A_2P_2 \leq 0
\end{equation}
Furthermore, \eqref{7} is identical to
\begin{equation}
B_1 = -A_1P_1C_1^T, \quad B_2=-A_2P_2C_2^T
\end{equation}
Thus $(A_1,B_1,C_1)$ and $(A_2,B_2,C_2)$ are IOHD systems with Hamiltonians $\frac{1}{2}x_1^TP_1^{-1}x_1, \frac{1}{2}x_2^TP_2^{-1}x_2$ (and are negative imaginary if $P>0$, and thus $P_1>0,P_2>0$).
\end{proof}

\smallskip

By decomposing $A_iP_i, i=1,2,$ into their skew-symmetric and symmetric parts as $A_iP_i = J_i -R_i, i=1,2,$ it follows that the positive feedback interconnection of $\Sigma_1$ and $\Sigma_2$ is alternatively given as the IOHD system with respect to the Poisson structure and resistive structure matrix
\begin{equation}
J= \begin{bmatrix} J_1 & 0 \\0 & J_2 \end{bmatrix}, \quad R= \begin{bmatrix} R_1 & 0 \\0 & R_2 \end{bmatrix},
\end{equation}
which is in general {\it different} from the originally assumed Poisson structure and resistive structure matrix (containing in general coupling terms between the two component systems). Hence Proposition \ref{prop:converse} stipulates that for an interconnected system which is IOHD the Poisson structure and resistive structure matrices can be always {\it redefined} as being block-diagonal, in which case, as a result, the coupling between the component systems $\Sigma_1$ and $\Sigma$ arises only from the coupling term in the interconnected Hamiltonian \eqref{interham1}.

Note finally that an essential part in the above proof is the fact that we consider a 'full' positive feedback interconnection \eqref{posfeedback}. In particular, a positive feedback interconnection of the form $u_1 = y_2 + e_1, \quad u_2 = y_1$ (no $e_2$ input) of two linear IOHD systems will result in an IOHD system with inputs $e_1$ and $y_1$, but the converse need not hold in this case.

\section{Networks of IOHD systems}
In this section it will be shown how the fact that the positive feedback interconnection of two IOHD systems is another IOHD system can be naturally extended to {\it networks} of IOHD systems. 

Consider an undirected graph with $N$ vertices and $M$ edges, together with an $N \times N$ weighted {\it adjacency matrix} $\A$ (symmetric since the graph is undirected). Furthermore, take the multi-agent point of view by associating to each of the vertices $i \in \{1,\cdots,N\}$ a nonlinear IOHD system \eqref{nonlinearham}, indexed by $i$, with equal number of inputs and outputs $m$ (independent of $i$). These IOHD systems are interconnected by setting
\begin{equation}
u =  \left( \A \otimes I_m \right) y + e,
\end{equation}
where $u$ is the stacked $Nm$ vector, with subvectors $u_1, \cdots, u_N$, and analogously for $y$ and the external inputs $e$. (Here $\A \otimes I_m$ denotes the Kronecker product of $\A$ and $I_m$, i.e., the $Nm \times Nm$ matrix obtained by multiplying every element of $\A$ by the $m \times m$ identity matrix $I_m$.)
Note that for the special case $N=2, M=1, \A_{12}=\A_{12}=1,$ this reduces to the positive feedback interconnection \eqref{posfeedback}. 
%with external inputs $e_1=e_2=0$.

As in Section II it can be readily seen that the resulting multi-agent system is again an IOHD system with inputs $e_1, \cdots, e_N$ and outputs $y_1, \cdots, y_N$, and total interconnected Hamiltonian given as
\begin{equation}\label{network}
\begin{array}{l}
H_{\mathrm{int}}(x_1,\cdots,x_N) :=  H_1(x_1) + \cdots + H_N(x_N) - \\[3mm]
\quad \quad \quad \frac{1}{2}\sum_{i,j=1}^N \A_{ij} (C_i(x_i))^TC_j(x_j) ,
\end{array}
\end{equation}
where $x_i,H_i,C_i, i=1, \cdots,N,$ are respectively the state vectors, Hamiltonians and output mappings of the IOHD systems associated to the vertices. Stability of the resulting IOHD system is again determined by this interconnected Hamiltonian.

Another scenario, more similar to the one considered in \cite{wang2015, xiong}, is to consider a {\it directed} graph, with $N$ vertices and $M$ edges and $N \times M$ incidence matrix $D$, and to associate not only to each of the vertices $i \in \{1,\cdots,N\}$ a nonlinear IOHD system \eqref{nonlinearham} with equal number of inputs and outputs $m$, but also to each of the edges $k \in \{1,\cdots,M\}$. These IOHD systems are now naturally interconnected by setting
\begin{equation}\label{posnetwork}
\begin{bmatrix}
u^v \\ u^e 
\end{bmatrix}
= \begin{bmatrix}
0 & D\otimes I_m \\ D^T\otimes I_m & 0 
\end{bmatrix}
\begin{bmatrix}
y^v \\ y^e 
\end{bmatrix} +
\begin{bmatrix}
e^v \\ e^e 
\end{bmatrix}
\end{equation}
Here $u^v$ and $y^v$ are the stacked $Nm$ vectors of inputs and outputs of the IOHD systems associated to the vertices, and $u^e$ and $y^e$ are the stacked $Mm$ vectors of inputs and outputs of the IOHD systems associated to the edges, and $e^v, e^e$ are the external inputs associated to the vertices, respectively edges. The resulting system is again an IOHD system, with total Hamiltonian being given by
\begin{equation}\label{network1}
\begin{array}{l}
H(x^v_1,\cdots,x^v_N, x^e_1,\cdots,x^e_M) :=  \\[3mm] H^v_1(x^v_1) + \cdots + H^v_N(x^v_N)  + 
H^e_1(x^e_1) + \cdots + H^e_M(x^e_M) \, - \\[3mm]
\quad \quad \sum_{i=1,k=1}^{N,M} D_{ik} (C^v_i(x^v_i))^TC^e_k(x^e_k)
\end{array}
\end{equation}
Here the superscripts ${}^v$ throughout refer to the vertex IOHD systems, and the superscripts ${}^e$ to the edge IOHD systems. In particular, $x^v_i,H^v_i,C^v_i, i=1, \cdots,N,$ are the state vectors, Hamiltonians and output mappings of the vertex IOHD systems, and $x^e_k,H^e_k,C^e_k, k=1, \cdots,M,$ are the state vectors, Hamiltonians and output mappings of the edge IOHD systems. 

Note that this last setting is {\it different} from the network interconnection of passive systems associated to the vertices and edges of a directed graph as considered in \cite{arcak} (or in the port-Hamiltonian case in \cite{siamgraphs}), since the interconnection \eqref{posnetwork} again corresponds to {\it positive} feedback.
%The set-up considered is as follows (partly motivated by \cite{wang2015} where a different network interconnection of negative imaginary systems is considered). \cite{xiong}

\section{Conclusions and outlook}
The class of input-output Hamiltonian systems with dissipation has been further explored as an interesting class of nonlinear systems, which is well-motivated by applications, and closely related to port-Hamiltonian systems and passivity(-based) control. A striking result is the converse result obtained (for linear systems) in Section III, which has the interesting implication that the Poisson and resistive structure matrices of an interconnected system that is IOHD can be always redefined in such a way that the coupling between the subsystems is only via the coupling term in the interconnected Hamiltonian. This is a case that happens quite frequently in modeling of multi-physics systems.

Of course, an extension of this converse result to the nonlinear case is a topic for further research, as well as the exploration of applications of the network interconnection theory of IOHD systems initiated in Section IV.


\begin{thebibliography}{99.}

\bibitem{abraham}
R. A. Abraham, J. E. Marsden.
\newblock {\it Foundations of mechanics} (2nd edition),
Benjamin/Cummings, Reading, Mass. 1978.

\bibitem{angeli0}
D. Angeli,
\newblock On systems with counter-clockwise input/output dynamics.
Proceedings {\it 43th IEEE Conference on Decision and Control}, Bahamas, Dec. 2004, pp. 2527--2532.

\bibitem{angeli1}
D. Angeli, 
\newblock Systems with counterclockwise input-output dynamics.
{\it IEEE Trans. Automatic Control}, 51(7), 1130--1143, 2006.

\bibitem{angeli2}
D. Angeli, 
\newblock Multistability in systems with counterclockwise input-output dynamics.
{\it IEEE Trans. Automatic Control}, 52(4), 596--609, 2007.

\bibitem{arcak}
M.~Arcak.
\newblock Passivity as a design tool for group coordination.
\newblock {\em Automatic Control, IEEE Transactions on}, 52(8):1380 --1390, 2007.

\bibitem{borja}
P. Borja, R. Cisneros, R. Ortega,
"Shaping the energy of port-Hamiltonian systems without solving PDE's",
In Proc. 54th {\it IEEE Conference on Decision and Control}, Osaka, Dec. 2015

\bibitem{brockett}
R.W. Brockett,
\newblock Control theory and analytical mechanics,
pp. 1--46 in {\it Geometric control theory}, eds. C. Martin, R. Hermann, Vol. VII of {\it Lie Groups: History, Frontiers and Applications}, Math Sci Press, Brookline, 1977.

\bibitem{crouch} P.E. Crouch, A.J. van der Schaft: {\it Variational and Hamiltonian control systems}.  Lectures Notes in Control and Inf. Sciences 101, Springer-Verlag, New York, 1987.
   
\bibitem{dalsmo}
M. Dalsmo, A.J. van der Schaft, 
``On representations and integrability of
mathematical structures in energy-conserving physical systems'', 
{\it SIAM J. Control and Optimization, vol.37}, pp. 54--91, 1999.

\bibitem{friswell}
M. Friswell, D.J. Inman,
"The relationship between positive position feedback and output feedback controllers",
{\it Smart Mater. Struct.}, vol. 8, pp. 285-291, 1999.

\bibitem{goh}
C.J. Goh, T.K. Caughey,
"On the stability problem caused by finite actuator dynamics in the control of large space structures",

\bibitem{jeltsema}
D. Jeltsema, R. Ortega, J.M.A. Scherpen, "An energy-balancing perspective of interconnection and damping assignment control of nonlinear systems", {\it Automatica} 40: 1643Ð1646, 2004.

\bibitem{cdc2011a}
F. Kerber, A.J. van der Schaft,
"Compositional properties of passivity",
50th {\it IEEE Conference on Decision and Control and European Control Conference} (CDC-ECC) Orlando, FL, USA, December 12-15, pp. 4628--4633, 2011.

\bibitem{lanzonpetersen1}
A. Lanzon, I.R. Petersen,
\newblock Stability robustness of a feedback interconnection of systems with negative imaginary frequency response,
{\it IEEE Trans. Automatic Control}, 53(4), 1042--1046, 2008.

\bibitem{bordeaux}
B. Maschke, A.J. van der Schaft, ``Port-controlled Hamiltonian systems: modelling
origins and system theoretic properties'', pp. 282-288 in {\it Proceedings 2nd IFAC Symposium on Nonlinear Control Systems (NOLCOS 2004)}, Ed. M. Fliess, Bordeaux, France, 
June 1992.

\bibitem{nvds}
H. Nijmeijer, A.J. van der Schaft, {\it Nonlinear Dynamical Control Systems},
Springer-Verlag, New York, 1990 (4th printing 1998), p. xiii+467. Corrected printing 2016.

\bibitem{ortega}
R.~Ortega, A.J. van~der Schaft, B.M. Maschke, and G.~Escobar.
\newblock Interconnection and damping assignment passivity-based control of
port-controlled Hamiltonian systems. {\em Automatica},
38:585--596, 2002. 

\bibitem{bernstein}
A. K. Padthe, JinHyoung Oh, D. S. Bernstein,
Counterclockwise dynamics of a rate-independent semilinear Duhem model,
Proceedings {\it 44th IEEE Conference on Decision and Control, and the European Control Conference}, Seville, Spain, December 12-15, 2005.

\bibitem{lanzonpetersen2}
I.R. Petersen, A. Lanzon,
\newblock Feedback control of negative-imaginary systems, 
{\it IEEE Control Systems Magazine}, 30(5), 54--72, 2010.

\bibitem{vanderschaftMST}
A.J. van der Schaft,
\newblock Hamiltonian dynamics with external forces and observations, 
{\it Mathematical Systems Theory}, 15, 145--168, 1982.

\bibitem{vanderschaft1}
A.J. van der Schaft, 
\newblock Observability and controllability for smooth nonlinear systems, 
{\it SIAM J. Control \& Opt.}, 20, pp. 338--354, 1982.

\bibitem{vanderschaftbook}
A.J. van der Schaft, {\it $L_2$-Gain and Passivity Techniques in
Nonlinear Control}, Lect. Notes in Control and Information
Sciences, Vol. 218, Springer-Verlag, Berlin, 1996, p. 168, 2nd
revised and enlarged edition, Springer-Verlag, London, 2000
(Springer Communications and Control Engineering series), p.xvi+249.

\bibitem{vanderschaftmaschkearchive}
A.J. van der Schaft, B.M. Maschke, The Hamiltonian formulation of energy conserving physical systems with external ports, {\it Archiv f\"{u}r Elektronik
 und \"{U}bertragungstechnik, 49}, pp. 362--371, 1995.
 
 \bibitem{siamgraphs}
A.J. van der Schaft, B.M. Maschke, "Port-Hamiltonian systems on graphs", 
{\it SIAM J. Control Optim.}, 51(2), 906--937, 2013.

 \bibitem{cdc2011}
 A.J. van der Schaft,
 "Positive feedback interconnection of Hamiltonian systems",
50th {\it IEEE Conference on Decision and Control and European Control Conference} (CDC-ECC) Orlando, FL, USA, December 12-15, pp. 6510--6515, 2011.

\bibitem{NOW}
A. van der Schaft, D. Jeltsema, 
"Port-Hamiltonian Systems Theory: An Introductory Overview," 
{\it Foundations and Trends in Systems and Control},
vol. 1, no. 2/3, pp. 173--378, 2014. 

\bibitem{venkatraman}
A. Venkatraman, A.J. van der Schaft, ÓEnergy shaping of port-Hamiltonian systems by using alternate passive input-output pairsÓ, {\it European Journal of Control}, 6:1Ð13, 2010.

\bibitem{wang-lanzon2015}
J. Wang, A. Lanzon, I.R. Petersen,
"Robust output feedback consensus for networked negative-imaginary systems",
{\it IEEE Trans. Automatic Control}, 60, pp. 2547--2552, 2015.

\bibitem{wang2015}
J. Wang, A. Lanzon, I.R. Petersen,
"Robust output feedback consensus for multiple heterogeneous negative-imaginary systems",
54th {\it IEEE Conference on Decision and Control} (CDC), Osaka, Japan, December 15-18, 2371--2376, 2015.

\bibitem{xiong}
J. Xiong, I.R. Petersen, A. Lanzon,
\newblock
"A negative imaginary lemma and the stability of interconnections of linear negative imaginary systems", 
{\it IEEE Trans. Automatic Control}, 55(10), 2342--2347, 2010.

%%%%%%%%%%%%%%%%%%%%%%%%%%%%%%%%%%%%%%%%%%%%%%%%%
%
%\bibitem{cortes} J. Cort\'es, A.J. van der Schaft, P.E. Crouch: Characterization of gradient control systems. {\sl SIAM J. Contr. and Optimiz.}, 44(4), pp. 1192-1214, 2005.

%\bibitem{Courant}
%T.J. Courant, ``Dirac manifolds'', {\it Trans. Amer. Math. Soc., 319}, pp. 631--661, 1990.
   
%\bibitem{dorfman}
%I.~Dorfman,
%{\em Dirac Structures and Integrability of Nonlinear Evolution
%  Equations}, John Wiley, Chichester, 1993.
%
%\bibitem{nvds}
%H. Nijmeijer, A.J. van der Schaft: {\it Nonlinear Dynamical Control Systems}, Springer, 1990.
%
%\bibitem{vanderschaftbook}
%A.J. van der Schaft, {\it $L_2$-Gain and Passivity Techniques in
%Nonlinear Control}, Lect. Notes in Control and Information
%Sciences, Vol. 218, Springer-Verlag, Berlin, 1996, 2nd edition, Springer-Verlag, London, 2000
%(Springer Communications and Control Engineering series).
%
%\bibitem{vanderschaftmaschkearchive}
%A.J. van der Schaft, B.M. Maschke, The Hamiltonian formulation of energy conserving physical systems with external ports, {\it Archiv f\"{u}r Elektronik
% und \"{U}bertragungstechnik, 49}, pp. 362--371, 1995.
% 
%\bibitem{vdssym}
%A.J. van der Schaft, ``Implicit Hamiltonian systems with symmetry'',
%{\it Rep. Math. Phys., 41}, pp. 203--221, 1998.
%

%\bibitem{wall}
%G.E. Wall,
%\newblock {\it Geometric properties of generic differentiable manifolds},
%Lect. Notes in Mathematics 597, Springer, Berlin, pp. 707--774, 1977.
\end{thebibliography}
\end{document}